\documentclass[a4paper]{article}

\usepackage[utf8]{inputenc}
\usepackage{lmodern}
\usepackage[T1]{fontenc}
\usepackage{microtype}
\usepackage{amsmath, amssymb, amsthm}
\usepackage{dsfont}

\numberwithin{equation}{section}

\newcommand{\Form}{\mathit{Form}}
\newcommand{\MRs}{\mathit{MR}}
\newcommand{\1}{\mathds{1}}
\newcommand{\K}{\mathds{K}}
\newcommand{\R}{{\mathds{R}}}
\newcommand{\C}{\mathds{C}}
\newcommand{\N}{\mathds{N}}
\newcommand{\A}{\mathcal{A}}
\newcommand{\Conv}{\mathcal{C}}
\renewcommand{\Re}{\operatorname{Re}}
\newcommand{\fra}{\mathfrak{a}}
\renewcommand{\mid}{\, \vert \,}
\newcommand\norm[1]{\lVert #1 \rVert}
\newcommand\abs[1]{\lvert #1 \rvert}

\theoremstyle{plain}
\newtheorem{theorem}{Theorem}[section]
\newtheorem{proposition}[theorem]{Proposition}
\newtheorem{lemma}[theorem]{Lemma}
\newtheorem{corollary}[theorem]{Corollary}

\theoremstyle{remark}

\allowdisplaybreaks
\title{Non-Autonomous Forms and Invariance}
\author{Dominik Dier}

\begin{document}

\maketitle

\begin{abstract}\label{abstract}
\noindent We generalize the Beurling--Deny--Ouhabaz criterion 
for parabolic evolution equations governed by forms to the non-autonomous, non-homogeneous and semilinear case.
Let $V, H$ are Hilbert spaces such that $V$ is continuously and densely embedded in $H$ and let $\A(t)\colon V\to V^\prime$ be the operator associated  with a 
bounded $H$-elliptic form $\fra(t,.,.)\colon V\times V \to \C$ for all $t \in [0,T]$.
Suppose $\Conv \subset H$ is closed and convex and $P \colon H \to H$ the orthogonal projection onto $\Conv$.
Given $f \in L^2(0,T;V')$ and $u_0\in \Conv$,
we investigate whenever the solution of the non-autonomous evolutionary problem
\[
u' (t)+\A(t)u(t)=f(t), \quad u(0)=u_0,
\]
remains in $\Conv$
and show that this is the case if
\begin{equation*}
	Pu(t) \in V \quad \text{and} \quad \Re \fra(t,Pu(t),u(t)-Pu(t)) \ge \Re \langle f(t), u(t)-Pu(t) \rangle
\end{equation*}
for a.e.\ $t \in [0,T]$.
Moreover, we examine necessity of this condition and apply this result to a semilinear problem.
\end{abstract}

\medskip


%%%%%%%%%%%%%%%%%%%%%%%%%%%%%%%%%%%%%%%%%%%%%%%%%%%%%%%%%%%%%%%%%%%%%%%%%%%%%%%%%%%%%%%%%%%%%%%%%%%
\section{Introduction}

The aim of the present article is to generalize the Beurling--Deny--Ouhabaz criterion 
for parabolic evolution equations to the non-autonomous, non-homogeneous and semilinear case.
We first consider an non-homogeneous Cauchy problem of the form
\begin{equation}\label{eq:introprob}
	u'(t)+ \A(t) u(t) = f(t) \text{ for a.e. } t \in (0,T), \quad u(0)=u_0.
\end{equation}
on the Hilbert space $V'$ where $V \overset d \hookrightarrow H$, $V'$ is the antidual of $V$
and the operators $\A(t)$ are associated with a bounded, quasi-coercive non-autonomous form 
$\fra \colon [0,T] \times V \times V \to \K$  (see Section~\ref{sec:prelim} for precise definitions).
If $f \in L^2(0,T;V')$ and  $u_0 \in H$, then by a classical theorem of Lions, there exists a unique solution
$u \in \MRs([0,T]) := H^1(0,T;V') \cap L^2(0,T;V)$.
Note that $V \overset d \hookrightarrow H$ and so $H \overset d \hookrightarrow V'$.
The maximal regularity space $\MRs([0,T])$ can be identified with a subspace of $C([0,T];H)$.
Thus the initial condition $u(0)=u_0$ is meaningful.

Let $\Conv \subset H$ be closed and convex and let $P \colon H \to H$ be the orthogonal projection onto $\Conv$.
Now if $\fra$ is \emph{autonomous}, i.e. $\fra(\cdot,v,w)$ is constant for every $v,w \in V$, and $f=0$, then the 
Beurling--Deny--Ouhabaz criterion states the following.
\begin{theorem}[{\cite[Theorem~2.1]{MVV05}, \cite[Theorem~2.2]{Ouh05}}]\label{thm:BDO}
	Every $u \in \MRs([0,T])$ with $u' +\A u=0$, $u(0) \in \Conv$ satisfies $u(t) \in \Conv$ for every $t \in [0,T]$
	if and only if $PV \subset V$ and $\Re \fra(Pv,v-Pv) \ge 0$ for every $v \in V$.
\end{theorem}

A first non-autonomous version of the result stated above was given in \cite{ADO14}:
Let $f \in L^2(0,T;V')$. Suppose that 
\begin{equation}\label{eq:introinv}
	PV \subset V \quad \text{and}\quad \Re \fra(t,Pv,v-Pv) \ge \Re\langle f(t), v-Pv \rangle \quad  (v \in V), 
\end{equation}
for a.e.\ $t\in[0,T]$,
then every $u \in \MRs([0,T])$ with $u'+\A u =f$, $u(0) \in \Conv$ satisfies $u(t) \in \Conv$ for every $t \in [0,T]$.

This result is very satisfying for the nonhomogeneous equation,
but is not suitable for semilinear equations (see below). For this reason we establish in Section~\ref{sec:invariance} a different criterion.
We show that $u \in \MRs([0,T])$ with $u'+\A u =f$, $u(0) \in \Conv$ satisfies $u(t) \in \Conv$ for every $t \in [0,T]$ provided that
\begin{equation*}
	Pu(t) \in V \quad \text{and} \quad \Re \fra(t,Pu(t),u(t)-Pu(t)) \ge \Re\langle f(t), u(t)-Pu(t) \rangle
\end{equation*}
for a.e.\ $t \in [0,T]$.
Thus we do not have to test for all $v\in V$ but just by $u(t)$ for a.e. $t \in [0,T]$.
This weaker condition makes a big difference when semilinear problems are considered (see Corollary~\ref{cor:invsemi}).
The crucial point in the proof is a version of the fundamental theorem of calculus involving the form and the projection under the natural regularity hypotheses (Lemma~\ref{lem:invariance}).
Its proof requires completely new ideas.

So far, in the non-autonomous and even in the autonomous non-homogeneous case (i.e.\ $f \neq 0$), nothing is known on necessity of \eqref{eq:introinv}.
In Section~\ref{sec:nec} we show under mild regularity assumptions that \eqref{eq:introinv} holds if every solution $u \in \MRs([t_0,b])$ of
$u'+\A u=f$, $u(t_0)\in \Conv$, $t_0 \in [a,b)$ has values in $\Conv$ (see Corollary~\ref{cor:pointwiseest}).

Also for the semilinear equation we prove necessity of the condition, however some more refined arguments are needed (see Section~\ref{sec:semilinnec}).

In Section~\ref{section:Example} we give an illustrating example of a semilinear equation with logistic growth.

\section*{Acknowledgment}
The author would like to thank Wolfgang Arendt for
his interest in this project and for many helpful comments.
%%%%%%%%%%%%%%%%%%%%%%%%%%%%%%%%%%%%%%%%%%%%%%%%%%%%%%%%%%%%%%%%%%%%%%%%%%%%%%%%%%%%%%%%%%%%%%%%%%%
\section{Preliminaries}\label{sec:prelim}

Let $\K$ be the field $\R$ or $\C$ and let $V$ and $H$ be Hilbert spaces over the field $\K$ such that $V \overset d \hookrightarrow H$; i.e., $V$ is continuously and densely embedded in $H$. 
Then $H \overset d \hookrightarrow V'$ via $v\mapsto (v \, \vert \, \cdot )_H$, where $V'$ denotes the antidual (or dual if $\K=\R$) of $V$.
Let $I:=[a,b]$, where $-\infty < a < b <\infty$ and let $\fra \colon I \times V \times V \to \K$. 
If $\fra(t, \cdot, \cdot)$ is sesquilinear for all $t \in I$ and $\fra(\cdot,v,w)$ is measurable for all $v,w \in V$, we say $\fra$ is a \emph{non-autonomous form}. 
If there exist a constant $M \ge 0$ such that
\begin{align}
	&\abs{\fra(t,v,w)} \le M \norm{v}_V \norm{w}_V  &(t \in I,\ v,w \in V), \label{eq:Vbounded}
	\intertext{we say that $\fra$ is \emph{bounded} and if there exist constants $\alpha >0$ and $\omega \in \R$ such that}
	&\operatorname{Re} \fra(t,v,v) + \omega \norm{v}_H^2 \ge \alpha \norm{v}_V^2 &(t \in I,\ v \in V). \label{eq:qcoercive}
\end{align}
we say that $\fra$ is \emph{$H$-elliptic} and \emph{coercive} in the case where $\omega =0$.
If $\fra$ is a bounded $H$-elliptic non-autonomous form we write $\fra \in \Form(I;V,H)$.

Let $\fra \in \Form(I;V,H)$.
For $t \in I$ we define $\A(t) \in \mathcal{L}(V,V')$ by $v \mapsto \fra(t,v, \cdot)$.
Then there exists an operator $\tilde \A \in \mathcal{L}(L^2(I;V); L^2(I;V'))$ such that $\tilde \A u(t) = \A(t) u(t)$ for a.e.\ $t \in I$.
By an abuse of notation we use the same letter for $\A$ and $\tilde \A$ and
we say that $\A$ is the \emph{operator associated with} $\fra$ and write $\A \sim \fra$.
In the separable case this is quite obvious,
but it needs special attention in the non-separable case (see \cite[Proposition~4.1]{DZ16}).

We define the \emph{maximal regularity space} $\MRs(I;V,H) := H^1(I;V') \cap L^2(I; V)$, with norm $\norm{u}_{\MRs(I;V,H)}^2 := \norm{u'}_{L^2(I;V')}^2+ \norm{u}_{L^2(I;V)}^2$.
Note that $\MRs(I;V,H)\hookrightarrow C(I;H)$, thus we consider $\MRs(I;V,H)$ as a subspace of $C(I;H)$.
Moreover, $\MRs(I;V,H)$ is a Hilbert space for the norm $\norm{\cdot}_{\MRs(I;V,H)}$. If no confusion occurs, we write $\MRs(I)$ instead of $\MRs(I;V,H)$.
A famous result due to J.~L.~Lions (see \cite[p.\ 513]{DL88}, \cite[p.~106]{Sho97}) 
establishes existence and uniqueness of the Cauchy problem associated with $\A$.
\begin{theorem}\label{thm:Lions}
Let $\fra \in \Form(I;V,H)$. 
Then for every $u_a\in H$ and $f \in L^2(I;V')$ there exists a unique
$u \in \MRs(I)$ such that
\begin{equation}\label{eq:CP}
	u'+ \A u = f, \quad u(a) = u_a
\end{equation}
Moreover, there exists a constant $c_\fra$ depending only on $M$, $\alpha$ and $\omega$ such that
\begin{equation}\label{eq:mrbound}
	\norm{u}^2_{\MRs(I)} \le c_\fra \left[ \norm{f}_{L^2(I;V')}^2 + \norm{u_a}_H^2 \right].
\end{equation}
\end{theorem}
Note that $u'$, $\A u$ and $f$ are in $L^2(I;V')$ so we consider the equality $u'+\A u = f$ in the space $L^2(I;V')$.
Thus we have \emph{maximal regularity} in the space $L^2(I;V')$.
In the following we call the function $u \in \MRs(I)$ the \emph{solution} of \eqref{eq:CP}.

\begin{lemma}[{\cite[p.~106]{Sho97}}]\label{lem:productrule}
	Let $u \in \MRs(I)$. Then $\norm{u}_H^2 \in W^{1,1}(I)$ with $(\norm{u}_H^2)' = 2 \Re \langle u' , u \rangle$.
\end{lemma}
As a consequence we obtain that
\[
	\norm{u(t)}_H^2-\norm{u(s)}_H^2 = 2 \int_s^t \Re \langle u' , u \rangle \ \mathrm{d}{t} \quad (s,t \in I),
\]
since $\MRs(I) \hookrightarrow C(I;H)$ and $u \mapsto \norm{u}_H^2$ is continuous.

%%%%%%%%%%%%%%%%%%%%%%%%%%%%%%%%%%%%%%%%%%%%%%%%%%%%%%%%%%%%%%%%%%%%%%%%%%%%%%%%%%
\section{Invariance of closed convex sets}\label{sec:invariance}

Let $I:=[a,b]$ where $-\infty < a < b <\infty$ and let $V, H$ be Hilbert spaces over the field $\K$ such that $V \overset d \hookrightarrow H$. 
Suppose $\fra \in \Form(I;V,H)$ and $\A \sim \fra$.
Let $\Conv \subset H$ be a closed convex set and let $P\colon H \to \Conv$ be the orthogonal projection onto $\Conv$;
i.e., for $x\in H$, $Px$ is the unique element in $\Conv$ such that
\begin{equation*}
	\norm{x - Px}_H \le \norm{x - y}_H \quad (y \in \Conv).
\end{equation*}

In this section we study invariance properties of the solution $u$ of \eqref{eq:CP} in terms of the form $\fra$ and the projection $P$. 
Our main result in this section is the following invariance criterion.
The point is that a criterion on an individual solution is given.

\begin{theorem}\label{thm:inv}
    Let $u_a \in \Conv$, $f \in L^2(I;V')$ and $u \in \MRs(I)$ be the solution of $u'+\A u =f$, $u(a) = u_a$.
    Then $u(t) \in \Conv$ for all $t \in I$ if
    \begin{equation}\label{eq:invcrit}
	Pu(t) \in V \quad \text{and}\quad \Re \fra(t, Pu(t), u(t)-Pu(t)) \ge  \Re \langle{f(t)}, {u(t)-Pu(t)}\rangle
    \end{equation}
    for a.e.\ $t \in I$.
\end{theorem}

As a corollary we obtain a result obtained in \cite[Theorem 2.2]{ADO14}
which is a criterion for invariance of all solutions.

\begin{corollary}\label{cor:invariance}
	Let $f \in L^2(I;V')$.
	Suppose that $PV \subset V$ and for a.e.\ $t \in I$
	\begin{equation}\label{eq:invariance_criterion}
		\Re \fra(t, Pv, v-Pv) \ge  \Re \langle{f(t)}, {v-Pv}\rangle \quad (v \in V).
	\end{equation}
	Then for every solution $u \in \MRs(I)$ of $u'+\A u =f$ with $u(a) \in \Conv$, we have $u(t) \in \Conv$ for every $t \in I$.
\end{corollary}
The advantage of Theorem~\ref{thm:inv} in contrast to Corollary~\ref{cor:invariance} is that we have to test merely by the solution itself. This becomes particularly important if we consider semilinear problems as the following criterion shows.
\begin{corollary}\label{cor:invsemi}
	Let $F \colon I \times \Conv \to V'$ be a function. Assume that $PV \subset V$ and 
	\begin{equation*}
        	\Re \fra(t, Pv, v-Pv) \ge  \Re \langle{F(t, Pv)}, {v-Pv}\rangle \quad (v \in V),
	\end{equation*}
	for a.e.\ $t \in I$.
	Let $u \in \MRs(I)$ such that $u'(t) +\A u(t) =F(t, Pu(t))$ for a.e.\ $t \in I$. If $u(a) \in \Conv$, then $u(t) \in \Conv$ for every $t \in I$.
	Consequently it is a solution of $u'(t)+\A u(t) = F(t, u)$, $t$-a.e.
\end{corollary}
\begin{proof}
	Let $f = u'+ \A u$. Then $f \in L^2(I;V')$ and \eqref{eq:invcrit} is satisfied for a.e.\ $t \in I$
	(but possibly not \eqref{eq:invariance_criterion}).
	The claim follows from Theorem~\ref{thm:inv}.
\end{proof}
As indicated in the proof, Corollary~\ref{cor:invsemi} cannot be deduced from Corollary~\ref{cor:invariance}.
In the remainder of this section we prove the following theorem,
which contains the assertion of Theorem~\ref{thm:inv}.
\begin{theorem}\label{thm:invstab}
    Let $u_a \in H$, $f \in L^2(I;V')$ and $u \in \MRs(I)$ be the solution of $u'+\A u =f$, $u(a) = u_a$.
    Suppose that
    \begin{equation}\label{eq:invcrit2}
        Pu(t) \in V \quad\text{and} \quad \Re \fra(t, Pu(t), u(t)-Pu(t)) \ge  \Re \langle{f(t)}, {u(t)-Pu(t)}\rangle
    \end{equation}
 	for a.e.\ $t \in I$.
   Then
	\[
		\norm{u(t)-Pu(t)}_H \le \norm{u_a-P u_a}_H e^{\omega (t-a)} \quad (t \in I),
	\]
	where $\omega \in \R$ satisfies $\Re \fra(t,v,v) \ge -\omega \norm{v}_H^2$ for every $t\in I$ and $v \in V$.
\end{theorem}
Note that such a constant $\omega$ always exists, since $\fra$ is $H$-elliptic. If $\fra$ is coercive, then $\omega$ can be chosen to be negative.
In that case the solution approaches $\Conv$ exponentially fast.
The following lemmas are crucial ingredients for the proof.

\begin{lemma}\label{lem:invbound}
	Let $v \in V$ such that $Pv \in V$ and $h \in V'$. Suppose that
	\begin{equation}\label{eq:lemPbound}
		\Re \fra(t,Pv,v-Pv) \ge \Re\langle h, v-Pv \rangle \quad (t \in I).
	\end{equation}
	Then 
	\[
		   \norm{v-Pv}_V^2 \le
            \tfrac 1{\alpha^2} \left( M^2\norm{v}^2_V + \norm{h}^2_{V'}  \right)+ \tfrac{2\omega}\alpha\norm{v-Pv}_H^2,
	\]
	where $M$, $\alpha >0$ and $\omega$ are constants such that \eqref{eq:Vbounded} and \eqref{eq:qcoercive} hold.
\end{lemma}
\begin{proof}
    By H-ellipticity and boundedness of $\fra$ and \eqref{eq:lemPbound} we have for $t \in I$
    \begin{multline*}
        \alpha \norm{v-Pv}_V^2 - \omega\norm{v-Pv}_H^2
            \le \Re \fra(t,v-Pv, v-Pv)\\
            \le \Re \fra(t,v,v-Pv)
                - \Re \langle{h}, {v-Pv}\rangle\\
           	\le M \norm{v}_V \norm{v-Pv}_V
                 + \norm{h}_{V'} \norm{v-Pv}_{V}
    \end{multline*}
    From this and the inequality $x y \le \frac \alpha {2} x^2 + \frac 1{2\alpha} y^2$, $x,y \in \R$,
    we obtain that 
    \begin{equation*}
        \frac\alpha 2\norm{v-Pv}_V^2 \le
            \frac 1{2\alpha} \left( M^2\norm{v}^2_V + \norm{h}^2_{V'}  \right)+ \omega\norm{v-Pv}_H^2. \tag*{\qedhere}
    \end{equation*}
\end{proof}

\begin{lemma}\label{lem:invariance}
    Let $u \in \MRs(I)$, such that $Pu \in L^2(I;V)$.
    Then for  $t,s \in I$ with $s \le t$ we have
    \[
        \norm{u(t)- Pu(t)}_H^2 - \norm{u(s)-Pu(s)}_H^2
            = 2 \int_s^t \Re \langle{u'(r)}, {u(r)-Pu(r)}\rangle \ \mathrm{d} r.
    \]
\end{lemma}
\begin{proof}
    	Recall that $\MRs(I) \hookrightarrow C(I;H)$
    	and consequently $Pu \in C(I;H)$, since $P\colon H \to H$ is a contraction.
    	Hence, it suffices to show $\norm{u-Pu}_H^2 \in W^{1,1}(I)$ with 
    	\[
    		(\norm{u-Pu}_H^2)' = 2\Re \langle u', u-Pu \rangle .
	\]
	Let $0<\delta<b-a$ and $h \in (a,b-\delta)$.
	We set for $t \in (a,b-\delta)$
	\begin{multline*}
		\theta_h(t) := \langle u(t+h)-u(t) , u(t+h)-Pu(t) \rangle
					+ \overline{\langle u(t+h)-u(t) , u(t)-Pu(t)\rangle}
	\end{multline*}
	and
	\begin{multline*}
		\eta_h(t) := \langle u(t+h)-u(t) , u(t+h)-Pu(t+h)\rangle\\ 
					+ \overline{\langle u(t+h)-u(t) , u(t)-Pu(t+h)\rangle}.
	\end{multline*}
	Since $\frac 1 h \int_.^{.+h} u'(s)\ \mathrm{d}{s} \to u'$ in $L^2((a,b-\delta);V')$ and $u(.+h) \to u$, $Pu(.+h) \to Pu$ in $L^2((a,b-\delta);V)$ as $h \to 0$
	, we have
	\[
		 \tfrac 1 h \theta_h,\tfrac 1 h \eta_h  \to  2\Re \langle u' , u-Pu \rangle \quad (h \to 0)
	\]
	in $L^1((a,b-\delta))$.
	Moreover, since $Pu(t)$ is the best approximation of $u(t)$ in $\Conv$ and $Pu(t+h)$ is the best approximation of $u(t+h)$ in $\Conv$, we have
	\begin{multline*}
		\eta_h(t) = \norm{u(t+h)-Pu(t+h)}^2_H - \norm{u(t)-Pu(t+h)}^2_H\\
			\le \norm{u(t+h)-Pu(t+h)}^2_H - \norm{u(t)-Pu(t)}^2_H\\
			\le \norm{u(t+h)-Pu(t)}^2_H - \norm{u(t)-Pu(t)}^2_H = \theta_h(t).
	\end{multline*}
	Let $\varphi \in C^\infty_c(a,b)$, $0<\delta<b-a$ such that $\operatorname{supp}\varphi \subset (a,b-\delta)$ and $h \in (0,\delta)$, then
	\begin{multline*}
		\frac 1 h \int_a^b \norm{u-Pu}_H^2 \left[ \varphi(t-h)-\varphi \right] \ \mathrm{d}{t}\\
			= \frac 1 h \int_a^b \left[ \norm{u(t+h)-Pu(t+h)}^2_H - \norm{u-Pu}^2_H \right] \varphi \ \mathrm{d}{t}\\
			\le \frac 1 h \int_a^b \left[ \theta_h \1_{\{\varphi \ge 0\}} + \eta_h \1_{\{\varphi < 0\}}\right] \varphi \ \mathrm{d}{t}
	\end{multline*}
	Now taking the limit $h \to 0$ shows 
	\begin{equation}\label{eq:diff_formula_inv}
		-\int_a^b \norm{u-Pu}_H^2  \varphi' \ \mathrm{d}{t} \le \int_a^b 2 \Re \left\langle u' , u -Pu \right\rangle \varphi \ \mathrm{d}{t}.
	\end{equation}
	Finally, if we replace $\varphi$ by $-\varphi$ we obtain equality in \eqref{eq:diff_formula_inv}.
\end{proof}

\begin{proof}[Proof of Theorem~\ref{thm:invstab}]
	Let $u_a \in H$, $f \in L^2(I;V')$ and $u \in \MRs(I)$ be the solution of $u'+\A u =f$, $u(a) = u_a$.
	Suppose that \eqref{eq:invcrit2} holds for a.e.\ $t \in I$.
	By Lemma~\ref{lem:invbound} and \eqref{eq:invcrit2} we obtain that $Pu \in L^2(I;V)$.
	Thus, by Lemma \ref{lem:invariance} for all $t \in I$ we have 
    	\begin{multline*}
		\norm{u(t)- Pu(t)}_H^2 -\norm{u_a- Pu_a}_H^2= 2 \int_a^t \Re \langle{u'}, {u-Pu} \rangle \ \mathrm{d} s\\
			= 2 \int_a^t \Re \langle{f-\A u}, {u-Pu}  \rangle \ \mathrm{d} s
			\le -2 \int_a^t  \Re \fra(s,u-Pu,u-Pu) \ \mathrm{d} s
    	\end{multline*}
    where we used the assumption $\eqref{eq:invcrit2}$ for  the inequality.
   Thus
    \[
        \norm{u(t)- Pu(t)}_H^2 \le \norm{u_a- Pu_a}_H^2+ 2 \omega \int_a^t \norm{u-Pu}_H^2 \ \mathrm{d} s \quad (t\in I).
    \]
	Now the claim of the theorem follows by Gronwall's lemma.
\end{proof}

%%%%%%%%%%%%%%%%%%%%%%%%%%%%%%%%%%%%%%%%%%%%%%%%%%%%%%%%%%%%%%%%%%%%%%%%%%%%%%%%%%%%%%%%%%%%%%%%%%%
\section{Necessity}\label{sec:nec}
Let $I:=[a,b]$ where $-\infty < a < b <\infty$ and let $V, H$ be Hilbert spaces over the field $\K$ such that $V \overset d \hookrightarrow H$. 
Suppose $\fra \in \Form(I;V,H)$, $\A \sim \fra$ and $f \in L^2(I;V')$.
Let $\Conv \subset H$ be a closed convex set and let $P\colon H \to \Conv$ be the orthogonal projection onto $\Conv$.

We say that $(\fra, f)$ is $\Conv$ \emph{invariant} if for every $c \in I$ and every $u \in \MRs([c,b])$ with $u'+\A u = f$, $u(c) \in \Conv$ 
we have $u(t) \in \Conv$ for all $t \in [c,b]$.
\begin{theorem}\label{thm:necInv}
	Suppose $(\fra, f)$ is $\Conv$ invariant.
	Then for every $u \in \MRs(I)$ with $u'+\A u =f$ we have
	$Pu \in L^2(I;V)$ and
	\begin{equation}\label{eq:necessity}
		 \Re \fra(t, Pu(t), u(t)-Pu(t)) \ge  \Re \langle{f(t)}, {u(t)-Pu(t)}\rangle \quad (\text{a.e.\ } t \in I).
	\end{equation}
\end{theorem}
Note that if $u(a) \in \Conv$ in the theorem above, then \eqref{eq:necessity} holds trivially, since $u(t)-Pu(t) =0$ for all $t \in [a,b]$.
But it is remarkable that \eqref{eq:necessity} holds for any initial value $u(a) \in H$.

Next we want to deduce a pointwise version from Theorem~\ref{thm:necInv}, which is in the spirit of the Beurling--Deny--Ouhabaz criterion.
Some regularity assumptions are needed for the proof. 
We say that $\fra$ is \emph{right-continuous} if $\lim_{t\downarrow c}\norm{\A(c)-\A(t)}_{\mathcal{L}(V,V')} = 0$ for every $c \in I$.
\begin{corollary}\label{cor:pointwiseest}
	Suppose that $\fra$ is right-continuous and that
	there exists a dense subspace $\tilde V$ of $V$, such that for every $c \in I$, $u_c \in \tilde V$ 
	the solution $u \in \MRs([c,b])$ of $u'+\A u = f$, $u(c)=u_c$ is in $C([c,b]; V)$.
	Then $(\fra, f)$ is $\Conv$ invariant if and only if $PV \subset V$ and 
	\begin{equation}\label{eq:pointwiseest}
		\Re \fra(t, Pv, v-Pv) \ge  \Re \langle{f(t)}, {v-Pv}\rangle \quad (\text{a.e.\ }t \in I,\, v \in V).
	\end{equation}
\end{corollary}
For example, if $f \in L^2(I;H)$ and $\fra$ is symmetric and of bounded variation, then for every $c \in I$ every solution $u \in \MRs([c,b])$ of $u'+ \A u = f$, $u_c \in V$ is in $C([c,b]; V)$ (see \cite{Die15}). Another example is the situation where $\fra$ is autonomous, i.e. $\fra(\cdot,v,w)$ is constant for every $v,w\in V$ and $f=0$. Then every solution of $u'+\A u = 0$, $u(a) \in D(A)$ is in $C(I;D(A))$, where $D(A) := \{v\in V: \A v \in H\}$, $\norm{v}_{D(A)}^2 = \norm{\A v}_H^2 + \norm{v}_H^2$ is densely embedded in $V$. Thus
we recover the Beurling--Deny--Ouhabaz criterion.

Recall that in Corollary~\ref{cor:pointwiseest} condition \eqref{eq:pointwiseest} is sufficient even if the additional regularity assumptions are not satisfied.
The author does not know whether the other implication is true without these assumptions.
\begin{proof}[Proof of Corollary~\ref{cor:pointwiseest}]
	Let $N_1 \subset I$ be a nullset such that $I \setminus N_1$ are Lebesgue points of $f$.
	Let $c\in [a,b)\setminus N_1$, $u_c \in \tilde V$ and $u \in \MRs([c,b])$ be the solution of $u'+ \A u = f$, $u(c)=u_c$.
	By Theorem~\ref{thm:necInv} we obtain that $Pu \in L^2([c,b];V)$ and that there exists a nullset $N_2 \subset [c,b]$ such that
	\begin{equation}\label{eq:pointwNec}
		\Re \fra(t, Pu(t), u(t)-Pu(t)) \ge  \Re \langle{f(t)}, {u(t)-Pu(t)}\rangle
	\end{equation}
	for $t \in [c,b]\setminus N_2$.
	Let $(t_n)_{n\in\N} \subset (c,b] \setminus N_2$ be a sequence such that $t_n \downarrow c$ and $f(t_n) \to f(c)$ in $V'$ for $n\to\infty$.
	Note that such a sequence exists by Lebesgue's differentiation theorem.
	By Lemma~\ref{lem:invbound} we obtain that $(Pu(t_n))_{n \in \N}$ is bounded in $V$, thus we conclude that
	that $Pu(c)\in V$ and $Pu(t_n) \rightharpoonup Pu(c)$ in $V$.
	Thus
	\begin{multline*}
		\Re \fra(c,u(c)-Pu(c), u(c)-Pu(c))\\
			\le \limsup_{n\to\infty} \Re \fra(t_n,u(t_n)-Pu(t_n), u(t_n)-Pu(t_n))\\
				\le \limsup_{n\to\infty} \Re \langle \A(t_n) u(t_n) - f(t_n), u(t_n)-Pu(t_n) \rangle\\
					=  \Re \langle \A(c) u(c) - f(c), u(c)-Pu(c) \rangle,
	\end{multline*}
	where we used that $\fra$ is right-continuous and $v \mapsto \Re\fra(c, v, v)+\omega \norm v_H^2$ is an equivalent norm on $V$ for the first inequality
	and \eqref{eq:pointwNec} with $t=t_n$ for the second inequality.
	This shows  $P\tilde V \subset  V$ and
	\begin{equation*}
		\Re \fra(t, Pv, v-Pv) \ge  \Re \langle{f(t)}, {v-Pv}\rangle \quad (t\in [a,b) \setminus N_1,\, v \in \tilde V).
	\end{equation*}
	Finally let $v \in V$ and $(v_n)_{n\in\N} \subset \tilde V$, $v_n \to v$ in $V$. 
	With a similar argument as above (where we replace the role of $u(t_n)$ by $v_n$ and $u(c)$ by $v$) we obtain the assertion of the corollary.
\end{proof}
We finish this section with the proof of the theorem.
\begin{proof}[Proof of Theorem~\ref{thm:necInv}]
	For $n \in \N$ and $k \in \{0,1, \dots, n\}$, let $t_{k}^n := a+ \frac{k}{n(b-a)}$.
	Let $v_{n,k} \in \MRs([t_{k-1}^n, t_{k}^n])$ 
	be the solution of $v_{n,k}'+ \A v_{n,k} = f$, $v_{n,k}(t_{k-1}^n)= Pu(t_{k-1}^n)$ and $v_n \in L^2(I; V)$,
	$v_n(t) := v_{n,k}(t)$ for $t \in [t_{k-1}^n, t_{k}^n)$, $k \in \{1, \dots, n \}$. 
	Let $k \in \{1,\dots,n\}$. Since $(\fra, f)$ is $\Conv$ invariant we have that $v_{n,k}(t) \in \Conv$ for all $t \in  [t_{k-1}^n,t_k^n]$. 
	Thus $\norm{u(t)-Pu(t)}_H \le \norm{u(t)-v_{n,k}(t)}_H$ for all $t \in [t_{k-1}^n,t_k^n]$.
	We set $\tilde u := u-Pu$, $\tilde v_{n,k} := u-v_{n,k}$ and $\tilde v_n := u-v_n$ and obtain
	by Lemma~\ref{lem:productrule}
	\begin{multline}\label{eq:boundforv_n}
		\norm{\tilde u(b)}^2_H-\norm{\tilde u(a)}^2_H = \sum_{k=1}^n \left(\norm{\tilde u(t_{k}^n)}^2_H-\norm{\tilde u(t_{k-1}^n)}^2_H\right)\\
		\le \sum_{k=1}^n \left(\norm{\tilde v_{n,k}(t_{k}^n)}^2_H-\norm{\tilde v_{n,k}(t_{k-1}^n)}^2_H\right)
		= \sum_{k=1}^n 2 \Re \int_{t_{k-1}^n}^{t_{k}^n} \langle \tilde v_{n,k}', \tilde v_{n,k} \rangle \ \mathrm{d}{s}\\
		= -\int_a^b 2\Re\fra(s, \tilde v_n, \tilde v_n) \ \mathrm{d}{s}.
	\end{multline}
	Suppose at first that $\tilde v_n \to \tilde u$ in $L^2(I; H)$.
	From \eqref{eq:boundforv_n} and $H$-ellipticity of $\fra$ we deduce that $\tilde u \in L^2(I; V)$, $\tilde v_n \rightharpoonup \tilde u$ in $L^2(I; V)$ and
	\begin{equation}\label{eq:boundforu}
		\norm{\tilde u(b)}^2_H-\norm{\tilde u(a)}^2_H \le
		 -\int_a^b 2\Re\fra(s, \tilde u, \tilde u) \ \mathrm{d}{s}.
	\end{equation}
	By \eqref{eq:boundforu} and Lemma~\ref{lem:invariance} we obtain
	\begin{multline*}
		-\int_a^b 2\Re\fra(s, \tilde u, \tilde u) \ \mathrm{d}{s} \ge \norm{\tilde u(b)}^2_H-\norm{\tilde u(a)}^2_H \\
			= 2 \int_a^b \Re \langle{u'}, {\tilde u}\rangle \ \mathrm{d} s 
			= 2 \int_a^b \Re \langle{f-\A u}, {\tilde u}\rangle \ \mathrm{d} s.
	\end{multline*}
	Hence
	\[
	 	0 \ge\int_a^b \Re \langle{f-\A P u}, {\tilde u}\rangle \ \mathrm{d} s.
	\]
	Note that this inequality holds also if we integrate over any interval $J \subset I$ instead of $I$ with a simple modification of the argument above.
	Applying Lebesgue's differentiation Theorem this finishes the proof if $\tilde v_n \to \tilde u$ in $L^2(I;H)$. We have
	\begin{multline*}
		\int_{t_{k-1}^n}^{t_{k}^n}\norm{\tilde v_n- \tilde u}^2_{H} \ \mathrm{d}{s}\\
			\le 3\int_{t_{k-1}^n}^{t_{k}^n}\norm{\tilde v_{n,k}- \tilde v_{n,k}(t_{k-1}^n)}^2_{H} \ \mathrm{d}{s} 
				+ 6\int_{t_{k-1}^n}^{t_{k}^n}\norm{u(t_{k-1}^n)- u}^2_{H} \ \mathrm{d}{s}\\
			\le \frac {3C} {n(b-a)} \left( \norm{\tilde v_{n,k}(t_{k-1}^n)}^2_{H} - \norm{\tilde v_{n,k}(t_{k}^n)}^2_{H} +\norm{\tilde v_{n,k}}_{L^2(t_{k-1}^n,t_{k}^n;H)}^2 \right)\\
				+\frac {6C}  {n(b-a)} \left( \norm{u(t_{k-1}^n)}^2_{H} - \norm{u(t_{k}^n)}^2_{H} + \norm{u}_{L^2(t_{k-1}^n,t_{k}^n;H)}^2 
					+ \norm{f}_{L^2(t_{k-1}^n,t_{k}^n;V')}^2 \right)
	\end{multline*}
	where we use that $Pu(t_{k-1}^n) = v_{n,k}(t_{k-1}^n)$ and that $P$ is a contraction in the first estimate and Lemma~\ref{lem:L^2convergency} below in the second estimate.
	We take the sum over $k$ from $1$ to $n$ and obtain by the first estimate of \eqref{eq:boundforv_n}
	\begin{multline*}
		\int_a^b\norm{\tilde v_n- \tilde u}^2_{H} \ \mathrm{d}{s} 
			\le\frac {3 C}  {n(b-a)} \Big( \norm{\tilde u(a)}^2_H-\norm{\tilde u(b)}^2_H +\norm{\tilde v_n}_{L^2(I;H)}^2\\
				 + 2\norm{u(a)}^2_H-2\norm{u(b)}^2_H + 2\norm{u}_{L^2(I;H)}^2 +2\norm{f}_{L^2(I;V')}   \Big).
	\end{multline*}
	By the reverse triangle inequality it follows that $\norm{\tilde v_n}_{L^2(I;H)}^2$ is bounded. Thus $\tilde v_n \to \tilde u$ in $L^2(I;H)$.
\end{proof}
\begin{lemma}\label{lem:L^2convergency}
	Let $u \in \MRs(I)$. Then there exists a constant $C>0$ such that
	\[
		\norm{u(t)-u(a)}^2_H \le C \left( \norm{u(a)}_H^2 - \norm{u(b)}_H^2 + \norm{u}_{L^2(I;H)}^2+ \norm{f}_{L^2(I;V')}^2 \right)
	\]
	for all $t \in I$, where $f:= u'+ \A u$.
\end{lemma}
\begin{proof}
	Let $t \in (a,b]$. We set $v(s) := u(\tfrac {1}{2}(t+s))-u(a+\tfrac {1}{2}(t-s))$.
	Then $v(a)=0$, $v(t)=u(t)-u(a)$ and $v \in \MRs([a,t])$.
	Thus
	\begin{multline*}
		\norm{u(t)-u(a)}^2_H = \norm{v(t)}^2_H-\norm{v(a)}^2_H = 2 \int_a^t \Re \langle v', v\rangle \ \mathrm{d}{s}\\
			\le \norm{v}_{\MRs([a,t])}^2 \le 2 \norm{u}_{\MRs([a,t])}^2\le 2 \norm{u}_{\MRs(I)}^2.
	\end{multline*}
	Moreover,
	\begin{multline*}
		\norm{u(b)}^2_H-\norm{u(a)}^2_H = 2 \int_a^b \Re \langle u', u\rangle \ \mathrm{d}{s}
			= 2 \int_a^b \Re \langle f-\A u, u\rangle \ \mathrm{d}{s}\\
		\le - 2\alpha \norm{u}_{L^2(I;V)}^2 + 2\omega \norm{u}_{L^2(I;H)}^2 + 2\norm{f}_{L^2(I;V')}\norm{u}_{L^2(I;V)}\\
		\le - \alpha \norm{u}_{L^2(I;V)}^2 + \omega \norm{u}_{L^2(I;H)}^2 + \frac 1 \alpha \norm{f}_{L^2(I;V')}^2
	\end{multline*}
	and
	\[
		\norm{u'}_{L^2(I;V')} = \norm{f-\A u}_{L^2(I;V')} \le \norm{f}_{L^2(I;V')} + M \norm{u}_{L^2(I;V)}.
	\]
	Now the claim follows by the three estimates above.
\end{proof}

%%%%%%%%%%%%%%%%%%%%%%%%%%%%%%%%%%%%%%%%%%%%%%%%%%%%%%%%%%%%%%%%%%%%%%%%%%%%%%%%%%%%%%%%%%%%%%%%%%%
\section{A semilinear problem}

In Section~\ref{sec:semilinnec} we want to study a semilinear version of the necessity conditions for invariance given in Section~\ref{sec:nec}.
Before that we want to establish well-posedness at least in a simple case.

Let $I:=[a,b]$ where $-\infty < a < b <\infty$ and let $V, H$ be Hilbert spaces over the field $\K$ such that $V \overset d \hookrightarrow H$. 
Let $\fra \in \Form(I;V,H)$, $\A \sim \fra$.
Suppose $F \colon I \times H \to V'$ satisfies $F(\cdot, v) \in L^2(I;V')$ for every $v \in H$ and there exists a constant $L>0$ such that
\begin{equation*}
	\norm{F(t,v)- F(t,w)}_{V'} \le L \norm{v-w}_H \quad (t\in I,\, v,w \in H).
\end{equation*}
\begin{proposition}\label{prop:semilinearsolution}
	For every $u_a \in H$ there exists a unique $u \in \MRs(I)$ such that $u'+\A u = F(\cdot, u)$, $u(a)=u_a$.
\end{proposition}
Before we prove Proposition~\ref{prop:semilinearsolution} we need several lemmas.
We denote by $M \ge 0$, $\alpha >0$ and $\omega \in \R$ the continuity and ellipticity constants in
\eqref{eq:Vbounded} and \eqref{eq:qcoercive}.

\begin{lemma}\label{lem:smallembedding}
	Let $u \in \MRs(I)$ with $u(a)=0$. Then
	\[
		\norm{u}^2_{L^2(I;H)} \le \frac {b-a}{2\sqrt2} \norm{u}_{\MRs(I)}^2.
	\]
\end{lemma}
\begin{proof}
	Since $u(a)=0$ we have
	\begin{multline*}
		\norm{u}^2_{L^2(I;V')} = \int_a^b \left\lVert\int_a^t u'(s) \ \mathrm{d}{s} \right\rVert^2_{V'} \ \mathrm{d}{t} \le  \int_a^b  \int_a^t  (t-a) \norm{u'(s)}^2_{V'} \ \mathrm{d}{s} \ \mathrm{d}{t}\\
			\le  \frac{(b-a)^2}2\int_a^b  \norm{u'(s)}^2_{V'} \ \mathrm{d}{s} = \frac{(b-a)^2}2\norm{u'}^2_{L^2(I;V')}.
	\end{multline*}
	Thus
	\begin{multline*}
		\norm{u}^2_{L^2(I;H)} \le \norm{u}_{L^2(I;V')} \norm{u}_{L^2(I;V)}\\
			 \le \frac{b-a}{\sqrt 2} \norm{u'}_{L^2(I;V')}\norm{u}_{L^2(I;V)} \le \frac{b-a}{2\sqrt 2} \norm{u}_{\MRs(I)}^2. \tag*\qedhere
	\end{multline*}
\end{proof}
\begin{lemma}\label{lem:glue}
	Let $-\infty <a<c<b<\infty$.
	Let $u_1 \in \MRs([a,c])$, $u_2 \in \MRs([c,b])$ such that $u_1(c)=u_2(c)$.
	Set $u(t):= u_1(t)$ for $t \in [a,c]$ and $u(t):= u_2(t)$ for $t \in (c,b]$.
	Then $u \in \MRs([a,b])$.
\end{lemma}

\begin{proof}[Proof of Proposition~\ref{prop:semilinearsolution}]
	a) By Theorem~\ref{thm:Lions} there exists a constant $c_\fra$ such that \eqref{eq:mrbound} holds.
	Let $b>a$ such that $(b-a)< q:=\frac{2\sqrt 2}{c_{\fra} L^2}$ and let $u_a \in H$.
	We define $S \colon \MRs([a,b]) \to \MRs([a,b])$, $v \mapsto u$,
	where $u$ is the solution of $u'+\A u =F(\cdot, v(\cdot))$, $u(a)=u_a$.
	For $u,v \in \MRs([a,b])$ we have by Lemma~\ref{lem:smallembedding}
	\begin{multline*}
		\norm{Su-Sv}_{\MRs([a,b])}^2 \le c_{\fra} \norm{F \circ u - F \circ v}_{L^2(a,b;V')}^2\\
			 \le c_{\fra} L^2 \norm{u - v}_{L^2(a,b;H)} \le c_{\fra} L^2 \frac{b-a}{2\sqrt 2} \norm{u - v}_{\MRs([a,b])}^2.
	\end{multline*}
	Thus $S$ is a strict contraction and by the Banach fixed-point theorem we obtain a unique $u \in \MRs([a,b])$ such that $Su=u$, i.e.\
	$u$ is the unique solution of $u'+\A u = F (\cdot, u(\cdot))$, $u(a)=u_a$.
	
	b) Part a) together with Lemma~\ref{lem:glue} yields a solution $u \in \MRs([a,b])$ for any $b>a$.
	
	c) We show uniqueness.
	Let $u_1, u_2 \in \MRs(I)$ be solutions with $u_1(a)=u_2(a)$.
	Recall that $u_1, u_2 \in C(I;H)$.
	Assume that $u_1$ and $u_2$ are different,
	then there exists $t_0 \in [a,b)$ such that $u_1=u_2$ on $[a,t_0]$ but
	$u_1(t_n) \neq u_2(t_n)$ for some $t_n \downarrow t_0$.
	Choose $0< \varepsilon < \min\{ q , T-t_0 \}$.
	Then there exist two different solutions on $[t_0, t_0 + \varepsilon]$ which contradicts a).
\end{proof}

%%%%%%%%%%%%%%%%%%%%%%%%%%%%%%%%%%%%%%%%%%%%%%%%%%%%%%%%%%%%%%%%%%%%%%%%%%%%%%%%%%%%%%%%%%%%%%%%%%%
\section{Semilinear Necessity}\label{sec:semilinnec}
In Section~\ref{sec:invariance} we saw that the invariance criterion for the non-homogeneous equation (Theorem~\ref{thm:inv}) could be applied immediately to semilinear problems.
The necessity result (Theorem~\ref{thm:necInv}) cannot so easily be carried over.
Additional arguments are needed to adapt the proofs of Section~\ref{sec:nec} to the semilinear case.
%Instead we have to modify the proofs of Section~\ref{sec:nec} to obtain a semilinear version.

Let $I:=[a,b]$ where $-\infty < a < b <\infty$ and let $V, H$ be Hilbert spaces over the field $\K$ such that $V \overset d \hookrightarrow H$. 
Let $\fra \in \Form(I;V,H)$, $\A \sim \fra$.
Let $F \colon I \times H \to V'$ be a function such that $F(\cdot, v) \in L^2(I;V')$ for every $v \in H$ and suppose that there exists a constant $L>0$ such that
\begin{equation}\label{eq:LipEst}
	\norm{F(t,v)- F(t,w)}_{V'} \le L \norm{v-w}_H \quad (t\in I,\, v,w \in H).
\end{equation}
Then by Proposition~\ref{prop:semilinearsolution},
for every $c\in [a,b)$ and every $u_c \in H$ there exists a unique $u \in \MRs([c,b])$ such that $u'+\A u = F(\cdot, u)$, $u(c)=u_c$.
	
Let $\Conv \subset H$ be a closed convex set and let $P\colon H \to \Conv$ be the orthogonal projection onto $\Conv$.
We say that $(\fra, F)$ is $\Conv$ \emph{invariant} if for every $c \in I$ and every $u \in \MRs([c,b])$ with $u'+\A u = F(\cdot, u)$, $u(c) \in \Conv$ 
we have $u(t) \in \Conv$ for all $t \in [c,b]$.

\begin{theorem}\label{thm:necInvSemi}
	Suppose $(\fra, F)$ is $\Conv$ invariant.
	Then for every $u \in \MRs(I)$ with $u'+\A u =F(\cdot, Pu)$ we have
	$Pu \in L^2(I;V)$ and
	\begin{equation*}
		 \Re \fra(t, Pu(t), u(t)-Pu(t)) \ge  \Re \langle{F(t,Pu)}, {u(t)-Pu(t)}\rangle \quad (\text{a.e.\ } t \in I).
	\end{equation*}
\end{theorem}
\begin{proof}
	For $n \in \N$ and $k \in \{0,1, \dots, n\}$, let $t_{k}^n := a+ \frac{k}{n(b-a)}$.
	Let $v_{n,k} \in \MRs([t_{k-1}^n, t_{k}^n])$ 
	be the solution of $v_{n,k}'+ \A v_{n,k} = F(\cdot, v_{n,k})$, $v_{n,k}(t_{k-1}^n)= Pu(t_{k-1}^n)$ and $v_n \in L^2(I; V)$,
	$v_n(t) := v_{n,k}(t)$ for $t \in [t_{k-1}^n, t_{k}^n)$, $k \in \{1, \dots, n \}$. 
	Let $k \in \{1,\dots,n\}$. Since $(\fra, F)$ is $\Conv$ invariant we have that $v_{n,k}(t) \in \Conv$ for all $t \in  [t_{k-1}^n,t_k^n]$. 
	Thus $\norm{u(t)-Pu(t)}_H \le \norm{u(t)-v_{n,k}(t)}_H$ for all $t \in [t_{k-1}^n,t_k^n]$.
	Let $\varepsilon \in (0,2)$. We set $\tilde u := u-Pu$, $\tilde v_{n,k} := u-v_{n,k}$ and $\tilde v_n := u-v_n$ and obtain
	\begin{multline}\label{eq:boundforv_n2}
		\norm{\tilde u(b)}^2_H-\norm{\tilde u(a)}^2_H = \sum_{k=1}^n \left(\norm{\tilde u(t_{k}^n)}^2_H-\norm{\tilde u(t_{k-1}^n)}^2_H\right)\\
		\le \sum_{k=1}^n \left(\norm{\tilde v_{n,k}(t_{k}^n)}^2_H-\norm{\tilde v_{n,k}(t_{k-1}^n)}^2_H\right)
		= \sum_{k=1}^n 2 \Re \int_{t_{k-1}^n}^{t_{k}^n} \langle \tilde v_{n,k}', \tilde v_{n,k} \rangle \ \mathrm{d}{s}\\
		= \int_a^b 2\Re\langle F(\cdot, u)-F(\cdot, v_n),\tilde v_n \rangle \ \mathrm{d}{s}-\int_a^b 2\Re\fra(s, \tilde v_n, \tilde v_n) \ \mathrm{d}{s}\\
		\le 2 L \int_a^b \norm{\tilde v_n}_H \norm{\tilde v_n}_V \ \mathrm{d}{s} -\int_a^b 2\Re\fra(s, \tilde v_n, \tilde v_n) \ \mathrm{d}{s}\\
			\le 	2 L \int_a^b \norm{\tilde v_n}_H \norm{\tilde v_n}_V \ \mathrm{d}{s} 
				-\varepsilon  \int_a^b \left(\alpha \norm{\tilde v_n}_V^2 -\omega \norm{\tilde v_n}_H^2 \right) \ \mathrm{d}{s} \\ 
					-\int_a^b (2-\varepsilon)\Re\fra(s, \tilde v_n, \tilde v_n) \ \mathrm{d}{s}\\
				\le \int_a^b \left(\varepsilon\omega + \tfrac{L^2}{\varepsilon\alpha} \right)\norm{\tilde v_n}_H^2  \ \mathrm{d}{s} -\int_a^b (2-\varepsilon)\Re\fra(s, \tilde v_n, \tilde v_n) \ \mathrm{d}{s}
	\end{multline}
	where we used Lemma~\ref{lem:productrule} in the equality in the second line, \eqref{eq:LipEst} and that $P$ is a contraction for the second inequality, 
	$H$-ellipticity of $\fra$ in the third inequality and Young's inequality ($2 xy \le \varepsilon \alpha x^2+ \frac 1 {\varepsilon \alpha} y^2,\, x,y \in \R$) in the last inequality.
	Suppose that $\tilde v_n \to \tilde u$ in $L^2(I; H)$, then \eqref{eq:boundforv_n2} yields
	$\tilde u \in L^2(I; V)$ with $\tilde v_n \rightharpoonup \tilde u$ in $L^2(I; V)$ and
	\begin{equation*}
		\norm{\tilde u(b)}^2_H-\norm{\tilde u(a)}^2_H \le
		 -\int_a^b (2-\varepsilon)\Re\fra(s, \tilde u, \tilde u) \ \mathrm{d}{s}.
	\end{equation*}
	Since $\varepsilon\in(0,2)$ was arbitrary we let $\varepsilon \downarrow 0$ and obtain
	\begin{equation}\label{eq:boundforu2}
		\norm{\tilde u(b)}^2_H-\norm{\tilde u(a)}^2_H \le
		 -\int_a^b 2\Re\fra(s, \tilde u, \tilde u) \ \mathrm{d}{s}.
	\end{equation}
	We combine \eqref{eq:boundforu2} and Lemma~\ref{lem:invariance}, thus
	\begin{multline*}
		-\int_a^b 2\Re\fra(s, \tilde u, \tilde u) \ \mathrm{d}{s} \ge \norm{\tilde u(b)}^2_H-\norm{\tilde u(a)}^2_H \\
			= 2 \int_a^b \Re \langle{u'}, {\tilde u}\rangle \ \mathrm{d} s 
			= 2 \int_a^b \Re \langle{F(\cdot, Pu)-\A u}, {\tilde u}\rangle \ \mathrm{d} s.
	\end{multline*}
	Hence
	\[
	 	0 \ge\int_a^b \Re \langle{F(\cdot, P u)-\A P u}, {\tilde u}\rangle \ \mathrm{d} s.
	\]
	Note that this inequality holds also if we integrate over any interval $J \subset I$ instead of $I$ with a simple modification of the argument above.
	This finishes the proof if $\tilde v_n \to \tilde u$ in $L^2(I;H)$. We have
	\begin{multline*}
		\int_{t_{k-1}^n}^{t_{k}^n}\norm{\tilde v_n- \tilde u}^2_{H} \ \mathrm{d}{s}\\
			\le 3\int_{t_{k-1}^n}^{t_{k}^n}\norm{\tilde v_{n,k}- \tilde v_{n,k}(t_{k-1}^n)}^2_{H} \ \mathrm{d}{s} 
				+ 6\int_{t_{k-1}^n}^{t_{k}^n}\norm{u(t_{k-1}^n)- u}^2_{H} \ \mathrm{d}{s}\\
			\le \frac {3C} {n(b-a)} \Big( \norm{\tilde v_{n,k}(t_{k-1}^n)}^2_{H} - \norm{\tilde v_{n,k}(t_{k}^n)}^2_{H} 
				+\norm{\tilde v_{n,k}}_{L^2(t_{k-1}^n,t_{k}^n;H)}^2 \\ + \norm{F(\cdot, Pu)-F(\cdot, v_{n,k})}_{L^2(t_{k-1}^n,t_{k}^n;V')}^2\Big)\\
				+\frac {6C}  {n(b-a)} \Big( \norm{u(t_{k-1}^n)}^2_{H} - \norm{u(t_{k}^n)}^2_{H} \\ + \norm{u}_{L^2(t_{k-1}^n,t_{k}^n;H)}^2 
					+ \norm{F(\cdot, Pu)}_{L^2(t_{k-1}^n,t_{k}^n;V')}^2 \Big),
	\end{multline*}
	where we use that $Pu(t_{k-1}^n) = v_{n,k}(t_{k-1}^n)$ and that $P$ is a contraction in the first estimate and Lemma~\ref{lem:L^2convergency} in the second estimate.
	We take the sum over $k$ from $1$ to $n$ and obtain by the first estimate of \eqref{eq:boundforv_n} and by \eqref{eq:LipEst} and the contractivity of $P$
	\begin{multline*}
		\int_a^b\norm{\tilde v_n- \tilde u}^2_{H} \ \mathrm{d}{s} 
			\le\frac {3 C}  {n(b-a)} \Big( \norm{\tilde u(a)}^2_H-\norm{\tilde u(b)}^2_H +(1+L)\norm{\tilde v_n}_{L^2(I;H)}^2\\
				 + 2\norm{u(a)}^2_H-2\norm{u(b)}^2_H + 2\norm{u}_{L^2(I;H)}^2 +2\norm{F(\cdot, Pu)}_{L^2(I;V')}   \Big).
	\end{multline*}
	By the reverse triangle inequality it follows that $\norm{\tilde v_n}_{L^2(I;H)}^2$ is bounded. Thus $\tilde v_n \to \tilde u$ in $L^2(I;H)$.
\end{proof}
%%%%%%%%%%%%%%%%%%
Next we also want to deduce a pointwise version in the semilinear setting, which is in the spirit of the Beurling--Deny--Ouhabaz criterion.
Again we use regularity assumptions as in Section~\ref{sec:nec}.
\begin{corollary}\label{cor:pointwiseestSemi}
	Suppose that $\fra$ and $F(\cdot, v)$, $v \in V$ are right-continuous and
	there exists a dense subspace $\tilde V$ of $V$, such that for every $c \in I$, $u_c \in \tilde V$ 
	the solution $u \in \MRs([c,b])$ of $u'+\A u = F(\cdot, Pu)$, $u(c)=u_c$ is in $C([c,b]; V)$.
	Then $(\fra, F)$ is $\Conv$ invariant if and only if $PV \subset V$ and 
	\begin{equation*}
		\Re \fra(t, Pv, v-Pv) \ge  \Re \langle{F(t,Pv)}, {v-Pv}\rangle \quad (t \in I,\, v \in V).
	\end{equation*}
\end{corollary}
For example, if  $F \colon I \times H \to H$ satisfies $F(\cdot, v) \in L^2(I;H)$ for every $v \in H$ and there exists a constant $L>0$ such that
\begin{equation*}
	\norm{F(t,v)- F(t,w)}_{H} \le L \norm{v-w}_H \quad (t\in I,\, v,w \in H).
\end{equation*}
and $\fra$ is of bounded variation and symmetric (see \cite{Die15}) or $\fra$ is Lipschitz continuous and satisfies $D(A^{1/2})=V$, where $A$ is the part of $\A(0)$ in $H$ (see \cite{ADLO14}), then the assumptions of the corollary above are satisfied.

Before we prove the corollary we state a simple, autonomous version of it.
The assumption $D(A^{1/2})=V$ is called \emph{Kato's square root property}.
For example by \cite{AT03} it is satisfied for elliptic operators in divergence form on Lipschitz domains with Dirichlet or Neumann boundary condition.

\begin{corollary}\label{cor:pointwiseest2}
	Let $F \colon H \to H$ be Lipschitz continuous. Suppose that $\fra$ is autonomous, i.e.\ $\fra(\cdot,v,w)$ is constant for all $v,w \in V$, and that $D(A^{1/2})=V$.
	Then $(\fra, F)$ is $\Conv$ invariant if and only if
	\[
		\Re \fra(Pv, v-Pv) \ge  \Re \langle{F(Pv)}, {v-Pv}\rangle \quad (v \in V).
	\]
\end{corollary}
\begin{proof}
	Let $c\in [a,b)$, $u_c \in \tilde V$ and $u \in \MRs([c,b])$ be the solution of $u'+ \A u = F(\cdot, u)$, $u(c)=u_c$.
	By Theorem~\ref{thm:necInv} we obtain that $Pu \in L^2([c,b];V)$ and that there exists a nullset $N \subset [c,b]$ such that
	\begin{equation}\label{eq:pointwNecSemi}
		\Re \fra(t, Pu(t), u(t)-Pu(t)) \ge  \Re \langle{F(t,Pu(t))}, {u(t)-Pu(t)}\rangle
	\end{equation}
	for $t \in [c,b]\setminus N$.
	Let $(t_n)_{n\in\N} \subset (c,b] \setminus N$ be a sequence such that $t_n \downarrow c$, then $F(t_n,u(t_n)) \to F(c, u(c))$ in $V'$ for $n\to\infty$.
	By Lemma~\ref{lem:invbound} we obtain that $(Pu(t_n))_{n \in \N}$ is bounded in $V$, thus we conclude that
	that $Pu(c)\in V$ and $Pu(t_n) \rightharpoonup Pu(c)$ in $V$.
	Thus
	\begin{multline*}
		\Re \fra(c,u(c)-Pu(c), u(c)-Pu(c))\\
			\le \limsup_{n\to\infty} \Re \fra(t_n,u(t_n)-Pu(t_n), u(t_n)-Pu(t_n))\\
				\le \limsup_{n\to\infty} \Re \langle \A(t_n) u(t_n) - F(t_n,Pu(t_n)), u(t_n)-Pu(t_n) \rangle\\
					=  \Re \langle \A(c) u(c) - F(c,Pu(c)), u(c)-Pu(c) \rangle,
	\end{multline*}
	where we used that $\fra$ is right-continuous and $v \mapsto \Re\fra(c, v, v)+\omega \norm v_H^2$ is an equivalent norm on $V$ for the first inequality
	and \eqref{eq:pointwNecSemi} with $t=t_n$ for the second inequality.
	This shows  $P\tilde V \subset  V$ and
	\begin{equation*}
		\Re \fra(t, Pv, v-Pv) \ge  \Re \langle{F(t,Pv)}, {v-Pv}\rangle \quad (t\in [a,b),\, v \in \tilde V).
	\end{equation*}
	Finally let $v \in V$ and $(v_n)_{n\in\N} \subset \tilde V$, $v_n \to v$ in $V$. 
	With a similar argument as above (where we replace the role of $u(t_n)$ by $v_n$ and $u(c)$ by $v$) we obtain the assertion of the corollary.
\end{proof}

%%%%%%%%%%%%%%%%%%%%%%%%%%%%%%%%%%%%%%%%%%%%%%%%%%%%%%%%%%%%%%%%%%%%%%%%%%%%%%%%%%%%%%%%%%%%%%%%%%%
\section{An illustrating example}\label{section:Example}

In this section we show by an example how the invariance criterion Theorem~\ref{thm:inv} can be applied.
We consider an elliptic operator of second order with time-dependent coefficients.  
Let $\Omega \subset \R^d$ be open and bounded, let $H$ be the real Hilbert space $L^2(\Omega)$ and let
$V$ be $H^1(\Omega)$ if we consider Neumann boundary conditions or $H_0^1(\Omega)$ for Dirichlet boundary conditions.
Let $T>0$.
We assume that $a_{jk} \colon [0,T] \times \Omega \to [0, \infty)$, $j,k \in \{ 1, \dots, d \}$ are measurable, bounded by some constant $M\ge 0$ and that there exists some constant $\alpha >0$ such that
\[
	\sum_{j,k=1}^d a_{jk}(t,x) \xi_j \xi_k \ge \alpha |\xi |^2 \quad (\xi = (\xi_1, ..., \xi_d) \in \R^d)
\]
for a.e.\ $(t,x) \in [0, T] \times \Omega$.
We define $\fra \colon [0,T] \times V \times V \to \R$ by
\[
	\fra(t,v,w) = \sum_{j,k=1}^d \int_\Omega a_{jk}(t,x) \partial_j v \partial_k w \ \mathrm{d} x.
\]
Then $\fra \in \Form([0,T];V,H)$, where $\eqref{eq:Vbounded}$ holds with the same constant $M$ and
\eqref{eq:qcoercive} holds with the same $\alpha$ and $\omega = \alpha$.

Let $F \colon \R \to \R$ be \emph{locally Lipschitz}, i.e.\ $F$ is Lipschitz continuous on bounded subsets of $\R$ and suppose that $F(0)=F(1)=0$.

\begin{proposition}
	For every $u_0 \in H$ with $u_0 \in [0,1]$ a.e., there exists a unique $u \in \MRs([0,T])$ such that $u(t) \in [0,1]$ a.e.\ for every $t \in [0,T]$ and
	\begin{equation}\label{eq:semilinearinv}
				u'(t)+ \A u(t) = F(u(t)) \quad (\text{a.e.\ }t\in[0,T]), \quad u(0)=u_0.
	\end{equation}
\end{proposition}
Note that $x \mapsto F(u(t)(x)) \in H$ for every $t \in [0,T]$ since $\Omega$ is bounded and $\MRs([0,T]) \hookrightarrow C([0,T];H)$.
Thus \eqref{eq:semilinearinv} makes sense.
For example, if $F(x)=x(1-x)$, then \eqref{eq:semilinearinv} is a diffusion equation with logistic groth.
\begin{proof}
	Let $\Conv := \{g \in H : g \in [0,1] \text{ a.e.}\}$. Then $\Conv$ is a closed and convex subset of $H$
	and the orthogonal projection $P \colon H \to H$ onto $\Conv$ is given by $Pg(x)=\max\{\min\{g(x),1\},0\} = \min\{\max\{g(x),0\},1\}$.
	Moreover, $g-Pg = (g-1)_+ - (-g)_+$.
	Thus $(F(Pg) \mid g-Pg)_H = 0$.
	Let $v \in V$, then $\nabla Pv = \nabla v \1_{\{ 0\le v \le 1 \}}$ and $\nabla (v-Pv) = \nabla v \1_{\{ v < 0\} \cup \{ v > 1\}}$.
	Thus $\fra(t, Pv, v-Pv) = 0$ for every $t \in [0,T]$.
	Finally, Proposition~\ref{prop:semilinearsolution} yields a unique solution $u\in \MRs([0,T])$ of
	\begin{equation*}
				u'(t)+ \A u(t) = F(Pu(t)) \quad (\text{a.e.\ }t\in[0,T]), \quad u(0)=u_0,
	\end{equation*}
	which is in $\Conv$ for ever $t \in [0,T]$ by Corollary~\ref{cor:invsemi}. Thus $Pu(t)=u(t)$ for every $t \in [0,T]$ and hence $u$ is our desired unique solution.
\end{proof}

%%%%%%%%%%%%%%%%%%%%%%%%%%%%%%%%%%%%%%%%%%%%%%%%%%%%%%%%%%%%%%%%%%%%%%%%%%%%%%%%%%%%%%%%%%%%%%%%%%%

\noindent
\emph{Dominik Dier}, Institute of Applied Analysis, 
University of Ulm, 89069 Ulm, Germany,
\texttt{dominik.dier@uni-ulm.de}

\end{document}